\documentclass{amsart}

\usepackage{amsmath}
\usepackage{amssymb}
\usepackage{amsthm}
\usepackage{mathrsfs}
\usepackage{stmaryrd}
\usepackage[all]{xy}
\usepackage[mathcal]{eucal}
\usepackage{verbatim}  
\usepackage[margin=1in]{geometry}  
\usepackage{enumitem}
\usepackage{url}

\newtheorem{theorem}{Theorem}
\newtheorem{proposition}[theorem]{Proposition}
\newtheorem{lemma}[theorem]{Lemma}
\newtheorem{question}{Question}

\newtheorem{corollary}[theorem]{Corollary}
\theoremstyle{definition}
\newtheorem{definition}[theorem]{Definition}

\newtheorem{remark}[theorem]{Remark}
\newtheorem{fact}[theorem]{Fact}
\newtheorem*{caveat}{Caveat}

\numberwithin{theorem}{section}
\numberwithin{question}{section}

\newcommand{\forces}{\Vdash}

\newcommand{\compat}{\parallel}

\DeclareMathOperator{\Ult}{\mathrm{Ult}}
\DeclareMathOperator{\cof}{\mathrm{cof}}
\DeclareMathOperator{\cf}{\mathrm{cf}}

\newcommand{\seq}[1]{\langle #1 \rangle}

\title[Destroying Saturation while Preserving Presaturation at an Inaccessible]{Destroying Saturation while Preserving Presaturation at an Inaccessible: An Iterated Forcing Argument}
\author{Noah Schoem}
\date{}

\begin{document}
\maketitle

\begin{abstract}
We prove that a large class of presaturated ideals at inaccessible cardinals can be de-saturated while preserving their presaturation, answering both a question of Foreman and of Cox and Eskew. We do so by iterating a generalized version of Baumgartner and Taylor's forcing to add a club with finite conditions along an inaccessible cardinal, and invoking Foreman's Duality Theorem.

\end{abstract}



\section{Introduction}
\label{intro}
It is a classical result of Solovay \cite{solovay_split} that the nonstationary ideal $NS_\kappa$ always has a $\kappa$-sized disjoint family of nonstationary sets; that is, in modern parlance, we say that $NS_\kappa$ is not $\kappa$-saturated.
One can argue Solovay's theorem using \emph{generic ultrapowers}. Suppose for sake of contradiction that$NS_\kappa$ is $\kappa$-saturated; then a $V$-generic filter $G$ for $(\mathcal{P}(\kappa)/NS_\kappa \setminus [\emptyset],\supseteq_{NS_\kappa})$ is a $V$-$\kappa$-complete $V$-normal $V$-ultrafilter with wellfounded ultrapower $\Ult(V,G)$.
Ultrapower arguments then yield a stationary set $S\subseteq\kappa$ in $V$ that is no longer stationary in $\Ult(V,G)$, hence is nonstationary in $V[G]$. 
But since $NS_\kappa$ was assumed to be $\kappa$-saturated, our forcing has the $\kappa$-chain condition and hence $S$ must be stationary in $V[G]$; this is a contradiction.

Solovay then asked whether $NS_\kappa$ is $\kappa^+$-saturated, and subsequent work by Gitik and Shelah in \cite{gitik_shelah_less_sat} showed that $NS_\kappa$ is not $\kappa^+$-saturated, except when $\kappa=\omega_1$.
Here, it is consistent (e.g. in the presence of Martin's Maximum, c.f. \cite{foreman_magidor_shelah_martin_max}) for $NS_{\omega_1}$ to be $\omega_2$-saturated.
Likewise, the nonstationary ideal on $\mathcal{P}_\kappa(\lambda)$ (for $\lambda\geq\kappa$) is known not to be $\kappa^+$-saturated unless $\kappa=\lambda=\omega_1$. 
This was due to Burke, Foreman, Gitik, Magidor, Matsubara, and Shelah; a summary and the proof of the case $\kappa=\lambda=\omega_1$ can be found in \cite{foreman_magidor_ns_ns}.

However, there are still useful arguments that can be written \emph{just} assuming that  $\mathcal{P}(\kappa) / NS_\kappa$ is \emph{precipitous}, i.e. induces a wellfounded ultrapower $\Ult(V,G)$.
For instance, this simplifies Silver's original argument in \cite{silver_sch} that if $\mathrm{SCH}$ fails at a singular cardinal, then the first singular cardinal at which $\mathrm{SCH}$ fails must have countable cofinality.

One can also ask whether there is \emph{any} ideal on $\kappa$ that is $\kappa$-saturated, $\kappa^+$-saturated, or even just precipitous. Results here are well-established and comprehensive.
The existence of exactly $\kappa$-saturated or $\kappa^+$-saturated ideals on inaccessible $\kappa$ are equiconsistent with a measurable cardinal.
This was first shown by Kunen and Paris in \cite{kunen_paris_measurables}, with weakly compact being compatible with $\kappa^+$-saturation (and it was known since early work of L\'evy and Silver that a $\kappa$-saturated ideal on $\kappa$ prevents $\kappa$ from being weakly compact).
Subsequently, Boos showed that an exactly $\kappa^+$-saturated ideal on $\kappa$ can exist at a non-weakly compact $\kappa$ in \cite{boos_efface_mahlo}.
As for successor cardinals, the consistency results are more striking. Certain arguments show that if $\kappa$ carries a $\kappa$-saturated ideal, then $\kappa$ must be weakly Mahlo, and hence not a successor.
Proofs can be found in \cite{baumgartner_taylor_wagon_ns_lcs} and \cite{ulam_dissertation}.
However, $\kappa^+$-saturated ideals \emph{can} occur at successor $\kappa$; the known ways to achieve this come from forcing over models with huge cardinals as done by Kunen in \cite{kunen_saturated_ideals} and Laver in \cite{laver_sat}.

Ideals on arbitrary sets $Z$ project downwards to subsets $Z'$ of $Z$, and it is natural to ask whether regularity of the inverse embedding implies nice saturation properties of the projected ideal:

\begin{question}[\cite{foreman_ideals_gen_ees}, Question 13 of Foreman]\label{foremanq13}
Let $n\in\omega$ and let $\mathcal{J}$ be an ideal on $Z\subseteq \mathcal{P}(\kappa^{+(n+1)})$.
Let $\mathcal{I}$ be the projection of $\mathcal{J}$ from $Z$ to some $Z'\subseteq \mathcal{P}(\kappa^{+n})$.
Suppose that the canonical homomorphism from $\mathcal{P}(Z')/\mathcal{I}$ to $\mathcal{P}(Z)/\mathcal{J}$ is a regular embedding.
Is $\mathcal{I}$ $\kappa^{+(n+1)}$-saturated?
\end{question}

The answer is no; prior work by Cox and Zeman in \cite{cox_zeman_antichain_catch} established counterexamples.
Later work by Cox and Eskew provided a template for finding counterexamples as follows.
We observe that $\mathcal{I}$ a $\kappa^{+n+1}$-saturated ideal on $\kappa^{+n}$ induces a wellfounded generic ultrapower and preserves $\kappa^{+n+1}$.
So we will say that an ideal $\mathcal{I}$ on $\kappa^{+n}$ is \emph{$\kappa^{+n+1}$-presaturated} if $\mathcal{I}$ induces a wellfounded generic ultrapower and preserves $\kappa^{+n+1}$.
Our template is then:

\begin{fact}[\cite{cox_eskew_kill_sat_save_presat}, corollary of Theorem 1.2]
Any $\kappa^{+n+1}$-presaturated, non-$\kappa^{+n+1}$ saturated ideal on $\kappa^{+n}$ provides a counterexample to Question \ref{foremanq13}.
\end{fact}

To construct such ideals for successor cardinals $\kappa=\mu^+$ (with $\mu$ regular and mild assumptions on cardinal arithmetic),
Cox and Eskew in \cite{cox_eskew_kill_sat_save_presat} generalized a forcing of Baumgartner and Taylor in \cite{baumgartner_taylor_sat_gen_ees_two} to add a club subset $C$ of $\kappa$ with $<\mu$-conditions.
(Baumgartner and Taylor's original version in \cite{baumgartner_taylor_sat_gen_ees_two} was for $\mu=\omega$.)
This $C$ prevented $\kappa^+$-saturated ideals on $\kappa$ from existing in the generic extension.
At the same time, their forcing was \emph{strongly proper}; with use of Foreman's Duality Theorem \cite{foreman_ideals_gen_ees}, a powerful tool for computing properties of ideals in generic extensions, Cox and Eskew were then able to argue that their forcing preserved the $\kappa^+$-presaturation of a large class of ideals (including $\kappa^+$-saturated ideals) in the generic extension.

This results in the following:
\begin{fact}
Let $V$ be a universe admitting $\kappa$-complete, $\kappa^+$-saturated ideals at $\kappa$ the successor of a regular cardinal. 
Then there is a forcing $\mathbb{P}$ such that $V^\mathbb{P}$ admits no $\kappa^+$-saturated ideals, however if $I$ is a $\kappa^+$-saturated ideal on $\kappa$ in $V$, then in $V^\mathbb{P}$ there is an $S\in \overline{I}^+$\footnote{where $\overline{I}$ is the ideal induced in $V^\mathbb{P}$ by $I$ and is defined by $\overline{I}=\{A\in \mathcal{P}^{V^\mathbb{P}}(\kappa) \mid \exists N\in I \ A\subseteq N\}$.}
such that $\overline{I}\upharpoonright S$ is $\kappa^+$-presaturated.
\end{fact}
It remained open as to whether the above could be done for $\kappa$ an inaccessible cardinal;
this was the content of Question 8.5 of \cite{cox_eskew_kill_sat_save_presat} and further clarifications provided in \cite{cox_schoem_overflow}.

This paper's central result establishes that Question \ref{foremanq13} is consistently false at $\kappa$ inaccessible, by way of partially extending the arguments and results of Theorem 4.1 of \cite{cox_eskew_kill_sat_save_presat}:

\begin{theorem}\label{bigthm}
Suppose $V$ is a universe of $\mathrm{ZFC}$ with an inaccessible cardinal $\kappa$ admitting $\kappa$-complete, normal, $\kappa^+$-saturated ideals on $\kappa$ concentrating on inaccessible cardinals below $\kappa$ (i.e. such that $Inacc_\kappa\in I^+$).
Then there is a poset $\mathbb{Q}$ such that:
\begin{enumerate}[label={(\roman*)}]
\item $V^\mathbb{Q}\models``\text{there are no }\kappa\text{-complete, }\kappa^+\text{-saturated ideals on }\kappa \text{  concentrating on inaccessible cardinals}"$ \label{bigthm:killsat}
\item If $I\in V$ is a $\kappa$-complete, normal, $\kappa^+$-saturated ideal on $\kappa$ concentrating on inaccessible cardinals, 
then $V^\mathbb{Q}\models``\overline{I}\text{ is }\kappa^+\text{-presaturated}"$ \label{bigthm:savepresat}
\end{enumerate}
where $\overline{I}=\{A\in \mathcal{P}^{V^\mathbb{Q}}(\kappa) \mid \exists N\in I \ A\subseteq N\}$.
\end{theorem}

We can further generalize Theorem \ref{bigthm}\ref{bigthm:savepresat} as follows:

\begin{theorem}\label{savemorepresat}
With the same assumptions, there is a $\mathbb{Q}$ such that if $\delta\geq\kappa$ is an inaccessible cardinal, $I\in V$ is normal, fine, precipitous, $\delta^+$-presaturated ideal of uniform completeness $\kappa$ on some algebra of sets $Z$ such that:
\begin{itemize}
\item $\mathcal{B}_I:=Z/I$ preserves the regularity of both $\kappa$ and $\delta$;
\item $\forces_{\mathcal{B}_I} \delta^+\leq|\dot{j}_I(\kappa)|=\delta<\dot{j}_I(\kappa)$
where $\dot{j}_I$ is a name for the generic elementary embedding $j_I:V\to M$;
\item $\mathcal{B}_I$ is $\delta^+$-proper on $IA_{<\delta^+}$;
\end{itemize}
then in $V^\mathbb{Q}$,
\begin{itemize}
\item $\overline{I}$ is not $\delta^+$-saturated
\item but $\overline{I}$ is $\delta^+$-presaturated
\end{itemize}
where $\overline{I}$ is as above.
\end{theorem}


Here, $IA_{<\delta}$ is the collection of internally approachable structures of length $<\delta$; we will give a precise definition later in Section \ref{prelims}.

\begin{remark}
It will turn out that the same $\mathbb{Q}$ will work for both Theorem \ref{bigthm} and Theorem \ref{savemorepresat}.
\end{remark}

\begin{remark}
In \cite{cox_eskew_kill_sat_save_presat}, the analogous theorem (Theorem 4.1(2)) argued that there is an $S\in \overline{I}^+$ such that $\overline{I}\upharpoonright S$ is not $\delta$-saturated, but it is $\delta$-presaturated.

The use of such an $S$ was required there due to the forcing involved not being $\kappa$-cc.
\end{remark}

This paper is structured as follows. 
Section \ref{prelims} presents the preliminary definitions and facts pertinent to this paper.
Section \ref{the-itern} introduces the forcing iteration $\mathbb{Q}$ of Theorems \ref{bigthm}\ref{bigthm:killsat}, \ref{bigthm}\ref{bigthm:savepresat}, and \ref{savemorepresat}.
Section \ref{killsat} shows that several saturated ideals are sundered from $V^{\mathbb{Q}}$.
Section \ref{savepresat} proves that a portion of presaturated posets remain presaturated in $V^{\mathbb{Q}}$.
Section \ref{open-qs} concludes and catalogs some conjectures.

\section{Preliminaries and notations}
\label{prelims}
Here are some definitions, theorems, and notations we use.

For a cardinal $\kappa$, we will write $Reg_\kappa$ for the set of regular cardinals below $\kappa$, and $\cof(\kappa)$ for the proper class of cardinals of cofinality $\kappa$.

If $\mathbb{P}$ is a notion of forcing in $V$, we will variously use $V^\mathbb{P}$ or $V[G]$ to refer to the generic extension of $V$ by $\mathbb{P}$.

We will further take for granted that the reader is familiar with forcing, iterated forcing, and ultrapowers.

\begin{definition}[ideals]
Let $\kappa$ be a cardinal.
An \emph{ideal} $I$ on $\kappa$ is a subset of $\mathcal{P}(\kappa)$ such that:
\begin{enumerate}
\item $\emptyset\in I$, $\kappa\notin I$
\item If $A\in I$ and $B\subseteq A$ then $B\in I$
\item If $A,B\in I$ then $A\cup B\in I$
\end{enumerate}

For $\mu\in Reg_\kappa$, the ideal $I$ is said to be \emph{$\mu$-complete} if whenever $\lambda\in Reg_\mu$ and $\seq{A_\alpha \mid \alpha<\lambda}\subseteq I$ then
\[\bigcup_{\alpha<\lambda} A_\alpha\]
is also in $I$.

The ideal $I$ is said to be \emph{normal} if whenever $\seq{A_\alpha\mid \alpha<\kappa}\subseteq I$, we have that the diagonal union
\[\nabla_{\alpha<\kappa} A_\alpha:=\{\beta<\kappa \mid \exists \alpha<\beta \ \beta\in A_\alpha\}\]
is also in $I$.
\end{definition}

An ideal is \emph{principal} if it contains a cofinite set; for our purposes, ideals are always assumed to be nonprincipal.

For an ideal $I$ on $\kappa$, we define $I^+:=\{S\subseteq \kappa \mid S\notin I\}$.

For example, $NS_\kappa$, the collection of nonstationary sets on $\kappa$, forms a normal ideal; its dual filter is the club filter on $\kappa$, and $(NS_\kappa)^+$ is the collection of stationary sets on $\kappa$.

\begin{definition}
If $I$ is an ideal on $\kappa$ then we may define an equivalence relation $\simeq_I$ on $\mathcal{P}(\kappa)$ by $A\simeq B$ if and only if $(A\setminus B)\cup (B\setminus A)\in I$.

We say that $A\leq_I B$ if $A\setminus B\in I$.

We may consider the equivalence classes $\mathcal{P}(\kappa)/I:=\{[A]_{\simeq_I}\mid A\subseteq \kappa\}$ as a poset with partial order $\leq_I$.
\end{definition}

Given $I$ an ideal on $\kappa$, we will write $\mathcal{B}_I:=\left(\mathcal{P}(\kappa)/I\right)\setminus [\emptyset]_{\simeq_I}$; when thinking of $\mathcal{B}_I$ as a poset, we will implicitly use the partial ordering $\leq_I$ and in many cases, $\mathcal{B}_I$ will be a separative notion of forcing (or even a complete Boolean algebra).

The above two definitions are Definitions 2.1, 2.17, and 2.18\footnote{In \cite{foreman_ideals_gen_ees}, a different version of normality is taken to be definitional, and the equivalence of these two versions is Proposition 2.19 of \cite{foreman_ideals_gen_ees}.} of \cite{foreman_ideals_gen_ees}.

The following definition summarizes some forcing properties of posets that will come in handy:

\begin{definition}[Chain condition, presaturation, and closure]
Let $(\mathbb{P},\leq)$ be a poset. We say that:
\begin{enumerate}[label={(\roman*)}]
\item (\cite{baumgartner_taylor_sat_gen_ees_two}, as Theorem 4.2) $\mathbb{P}$ is \emph{$\mu$-presaturated} if for every $\lambda<\mu$ and every family $\seq{A_\alpha \mid \alpha<\lambda}$ of antichains,
there are densely many $p\in\mathbb{P}$ such that for all $\alpha$, $\{q\in A_\alpha \mid p\compat q\}$ has cardinality $<\mu$. Note that $\mu$-cc implies $\mu$-presaturation.

\item $\mathbb{P}$ is \emph{$<\kappa$-closed} if whenever $\tau<\kappa$ and $\seq{p_\alpha \mid \alpha<\tau}$ is a $\leq$-decreasing sequence in $\mathbb{P}$, there is a $p\in\mathbb{P}$ such that $p\leq p_\alpha$ for all $\alpha<\tau$

\item $\mathbb{P}$ is \emph{$<\kappa$-directed closed} ($<\kappa$-dc) if whenever $D\subseteq \mathbb{P}$ is a directed set\footnote{that is, for all $p,q\in D$, there is an $r\in D$ such that $r\leq p,q$} with $|D|<\kappa$, there is a $q\in \mathbb{P}$ such that whenever $p\in D$, $q\leq p$

\item $\mathbb{P}$ is \emph{$\mu$-preserving} (for $\mu$ a $V$-cardinal) if $V^\mathbb{P}\models ``\check{\mu}\text{ is a cardinal}"$
\end{enumerate}
\end{definition}

Some of these properties have analogues for ideals as well.
To begin with, if $I$ is a $\kappa$-complete ideal on $\kappa$ then a $\mathcal{B}_I$-generic object induces a $V$-$\kappa$-complete ultrafilter $U$ on $\kappa$ from which the ultrafilter $\Ult(V,U)$ may be formed.
If $I$ is additionally normal, then $U$ will be $V$-normal.

\begin{definition}[Definition 2.4 of \cite{foreman_ideals_gen_ees}]
An ideal $I$ is said to be \emph{precipitous} if whenever $U$ is a $\mathcal{B}_I$-generic object over $V$,
$\Ult(V,U)$ is well-founded.
\end{definition}

For $I$ an ideal on $\kappa$, we will say that $I$ is defined to be \emph{$\mu$-saturated} if $\mathcal{B}_I$ has the $\mu$-cc.
It is well known that a $\mu$-saturated ideal is precipitous and $\mu$-preserving, i.e. that $V^{\mathcal{B}_I}\models``\mu\text{ is a cardinal}"$.
However, we make the following caveat:

\begin{caveat}
In contrast to the above definition, an ideal $I$ is defined to be $\mu$-presaturated if $I$ is precipitous and $\mu$-preserving.

Only for precipitous ideals $I$ is it the case that $I$ is $\mu$-presaturated if and only if $\mathcal{B}_I$ is $\mu$-presaturated as a forcing poset.
\end{caveat}

Presaturation can be pushed downwards through an iteration:

\begin{lemma}[Lemma 2.12 of \cite{cox_eskew_kill_sat_save_presat}]\label{pushpresatdown}
If $\mathbb{P} * \dot{\mathbb{Q}}$ is $\kappa$-presaturated then $\mathbb{P}$ is $\kappa$-presaturated and $1_\mathbb{P}\forces \dot{\mathbb{Q}}$ is $\kappa$-presaturated.
\end{lemma}

Whether the converse holds is currently an open problem; this appears as Question 8.6 of \cite{cox_eskew_kill_sat_save_presat}. 

Next we go over the notion of properness and relate properness and closedness to presaturation.
Let $\delta$ be regular uncountable, and let $H\supsetneq\delta$.
Then we write $\mathcal{P}_\delta(H)$ for all subsets of $H$ of size $<\delta$, and $\mathcal{P}^*_\delta(H)$ to denote the set of all $x\in \mathcal{P}_\delta(H)$ such that $x\cap\delta\in\delta$.

\begin{definition}
Let $\mathbb{P}$ be a notion of forcing, $\theta$ sufficiently large so that $\mathbb{P}\in H_\theta$, and $M\prec (H_\theta,\in,\mathbb{P})$. 

We say that $p\in\mathbb{P}$ is an \emph{$(M,\mathbb{P})$-master condition} if for every dense $D\in M$, $D\cap M$ is predense below $p$; equivalently, $p\forces_{\mathbb{P}} M[\dot{G}_{\mathbb{P}}]\cap V=M$. 

Additionally, we say that $p$ is an \emph{$(M,\mathbb{P})$-strong master condition} if for every $p'\leq p$, there is some $p'_M\in M\cap\mathbb{P}$ such that every extension of $p'_M$ in $M\cap\mathbb{P}$ is compatible with $p'$. 
\footnote{It is straightforward to see that strong master conditions are also master conditions.}

Further, $\mathbb{P}$ is \emph{(strongly) proper with respect to $M$} if every $p\in M\cap\mathbb{P}$ has a $q\leq p$ such that $q$ is an $(M,\mathbb{P})$-(strong) master condition.

We say that $\mathbb{P}$ is \emph{(strongly) $\delta$-proper on a stationary set} if there is a stationary subset $S$ of $\mathcal{P}^*_\delta(H_\theta)$ such that for every $M\in S$, $M\prec (H_\theta,\in,\mathbb{P})$ and $\mathbb{P}$ is (strongly) proper with respect to $M$.
\end{definition}

Note that $\{M\in \mathcal{P}^*_\delta(H_\theta)\mid M\prec (H_\theta,\in,\mathbb{P})\}$ is a club subset of $\mathcal{P}^*_\delta(H_\theta)$; so a forcing being $\delta$-proper on a stationary set really only depends on the properness condition.

\begin{fact}
If $\mathbb{P}$ is $\delta$-proper on a stationary set, then $\mathbb{P}$ is $\delta$-presaturated.
\end{fact}

This fact appears as Fact 2.8 of \cite{cox_eskew_kill_sat_save_presat}, with proof; their proof, in turn, generalizes a result of Foreman and Magidor in the case of $\delta=\omega_1$ (namely, Proposition 3.2 of \cite{foreman_magidor_lcs_ctx_ch}).

For the posets we will be working with, we will have a specific stationary subset witnessing $\delta$-properness:

\begin{definition}
For $\delta$ regular and $\theta>>\delta$, we say that $IA_{<\delta}\subseteq \mathcal{P}^*_\delta(H_\theta)$, 
the ``internally approachable sets of length $<\delta$", 
is the collection of all $M\in \mathcal{P}^*_\delta(H_\theta)$, with $|M|=|M\cap \delta|$, that are \emph{internally approachable},
i.e. such that there is a $\zeta<\delta$ and a continuous $\subseteq$-increasing sequence $\seq{N_\alpha\mid \alpha<\zeta}$ whose union is $M$, such that $\vec{N}\upharpoonright \alpha\in M$ for all $\alpha<\zeta$.
\end{definition}

In a sense, internal approachability is preserved by any generic extension:

\begin{fact}\label{principle-A}
Suppose $\mathbb{P}$ is a poset, $M\prec(H_\theta,\in,\mathbb{P})$, $\seq{N_\alpha\mid\alpha<\zeta}$ witnesses that $M\in IA_{<\delta}$, and $G$ is $(V,\mathbb{P})$-generic.
Then in $V[G]$, $\seq{N_\alpha[G]\mid \alpha<\zeta}$ witnesses that $M[G]\in IA_{<\delta}$. (Without loss of generality, we may assume that $\mathbb{P}\in N_0$.)
\end{fact}

It is a standard fact that $IA_{<\delta}$ is stationary.
The following lemma makes clear its utility:

\begin{lemma}\label{lemma-iaposets}
Let $\delta$ be regular and uncountable. Then:
\begin{enumerate}[label={(\roman*)}]
\item\label{lemma-iaposets:cc}
If $\mathbb{P}$ is $\delta$-cc and $M\prec (H_\theta,\in,\mathbb{P})$ is an element of $\mathcal{P}^*_\delta(H_\theta)$ (i.e. if $M\cap\delta\in\delta$), then $1_\mathbb{P}$ is an $(M,\mathbb{P})$-master condition; in particular $\mathbb{P}$ is $\delta$-proper on $\mathcal{P}^*_\delta(H_\theta)$.

\item\label{lemma-iaposets:closed}
If $\mathbb{Q}$ is $<\delta$-closed then $\mathbb{Q}$ is $\delta$-proper on $IA_{<\delta}$.

\item\label{lemma-iaposets:iter}
If $\mathbb{P}$ is $\delta$-proper on $IA_{<\delta}$ and $\forces_\mathbb{P}``\dot{\mathbb{Q}}\text{ is }\delta\text{-cc}$ or $\forces_\mathbb{P}``\dot{\mathbb{Q}}\text{ is }<\delta\text{-closed}$ then $\mathbb{P}*\mathbb{\dot{Q}}$ is $\delta$-proper on $IA_{<\delta}$.
\end{enumerate}
\end{lemma}

This is roughly Fact 2.9 out of \cite{cox_eskew_kill_sat_save_presat}. The following proof is largely reproduced from \cite{cox_eskew_kill_sat_save_presat} as well.

\begin{proof}
For part \ref{lemma-iaposets:cc}, let $A\in M$ be a maximal antichain in $\mathbb{P}$.
Since $|A|<\delta$ and $M\cap\delta\in\delta$, we have that $A\subseteq M$.
Thus $1_\mathbb{P}\forces M[\dot{G}]\cap\check{V}=M$, so $1_\mathbb{P}$ is a master condition for $M$.

Part \ref{lemma-iaposets:closed} is due to Foreman and Magidor in \cite{foreman_magidor_lcs_ctx_ch}. 

As for part \ref{lemma-iaposets:iter}, let $G$ be $\mathbb{P}$-generic over $V$.
Suppose that $M\prec (H_\theta,\in,\mathbb{P}*\dot{\mathbb{Q}})$ and $M\in IA_{<\delta}$.
By Fact \ref{principle-A}, combined with \ref{lemma-iaposets:cc} and \ref{lemma-iaposets:closed},
$\mathbb{P}$ forces that $\dot{\mathbb{Q}}$ is proper with respect to $M[\dot{G}]$.
Hence $\mathbb{P}*\dot{\mathbb{Q}}$ is proper with respect to $M$.
\end{proof}

Presaturation comes with some partial cofinality preservation:
\begin{fact}
If $\mathbb{P}$ is $\lambda$-presaturated for $\lambda$ regular then 
\[\forces_{\mathbb{P}} \cof^V(\geq\lambda)=\cof^{V[\dot{G}]}(\geq\lambda)\]
\end{fact}
The above fact has a partial converse. We will not make use of it, but it is another known way to argue that certain iterations of presaturated forcings are presaturated:
\begin{fact}\label{cof->presat}
If $\mathbb{P}$ is $\lambda^{+\omega}$-cc for some regular $\lambda\geq\omega_1$ and
\[\forall n\in\omega \ \forces_{\mathbb{P}} cf^{V[\dot{G}]}\left(\left(\lambda^{+n}\right)^V\right)\geq\lambda\]
then $\mathbb{P}$ is $\lambda$-presaturated.
\end{fact}

This appears as Fact 2.11 in \cite{cox_eskew_kill_sat_save_presat}, which in turn is a generalization of Theorem 4.3 of \cite{baumgartner_taylor_sat_gen_ees_two}.

\begin{fact}\label{foreman_ideals_gen_ees:ultrapower-closure}
For a $\kappa$-complete, $\kappa^+$-saturated ideal $I\in V$,
if $U$ is a $\mathcal{B}_I$-generic filter over $V$
then in $V[U]$, $^\kappa \Ult(V,U)\subseteq \Ult(V,U)$; that is, $\Ult(V,U)$ is closed under $\kappa$-sequences from $V[U]$.
\end{fact}

This follows from Propositions 2.9 and 2.14 of \cite{foreman_ideals_gen_ees}.

We will sometimes write $\Ult(V,I)$ to denote $\Ult(V,U)$, and will also write $j_I$ to denote $j_U:V\to \Ult(V,U)$.

If $I\in V$ is an ideal on $\kappa$ and $\mathbb{P}$ is a notion of forcing understood from context, then we will write $\overline{I}:=\left\{N\in \mathcal{P}^{V^\mathbb{P}}(\kappa)\mid \exists A\in I \ N\subseteq A\right\}$.

The following two simplified versions of Foreman's Duality Theorem will be useful later:

\begin{lemma}\label{foreman_ideals_gen_ees:groundideals}
For a $\kappa$-complete, $\kappa^+$ saturated $I\in V$, $\overline{I}$ is $\kappa^+$-saturated in $V^\mathbb{Q}$ if and only if $\forces_{\mathcal{B}_I} \dot{j}_I(\mathbb{Q})$ is $\kappa^+$-cc.
\end{lemma}
 
This appears as Corollary 7.21 in \cite{foreman_ideals_gen_ees}.

\begin{theorem}
\label{easyduality}
Let $I$ be a $\kappa$-complete normal precipitous ideal in $V$ and $\mathbb{Q}$ be a $\kappa$-cc poset. 
Then $\overline{I}$ is precipitous and there is a canonical isomorphism witnessing that
\[\mathcal{B}\left(\mathbb{Q} * \mathcal{B}_{\overline{I}}\right)\cong \mathcal{B}\left(\mathcal{B}_I * \dot{j}_I\left(\mathbb{Q}\right)\right)\]
where $\mathcal{B}(\mathbb{P})$ refers to the Boolean completion of $\mathbb{P}$.
\end{theorem}

This statement appears in \cite{cox_eskew_kill_sat_save_presat} as Fact 2.24, and is a corollary of Theorem 7.14 of \cite{foreman_ideals_gen_ees}.

\section{The Forcing Iteration}
\label{the-itern}
Through the rest of this paper, fix $\kappa$ to be a Mahlo cardinal.
We do this because by assuming that $V$ admits a $\kappa$-complete, normal, $\kappa^+$-saturated ideal on $\kappa$ concentrating on regular cardinals below $\kappa$, we have that $\kappa$ is Mahlo.

Over cardinals below $\kappa$, we will define a forcing iteration that will destroy $\kappa^+$-saturation but preserve $\kappa^+$-presaturation for ideals on $\kappa$, concentrating on regulars by adding, for each $\mu<\kappa$, $\mu$ inaccessible, a club subset $C_\mu$ of $\mu^+$ using $<\mu$-conditions.
This club $C_\mu$ will fail to contain certain ground model sets, 
in the sense that if $X\in V$ and $|X|\geq \mu$ then $X\not\subseteq C_\mu$.

Although our iteration will only be active at inaccessibles, the forcing $\mathbb{P}(\mu)$ to add $C_\mu$ is well-behaved under milder cardinal arithmetic assumptions,
so we define $\mathbb{P}(\mu)$ in this more general context:

\begin{definition}\label{defpmu}
  Let $\mu<\kappa$ be a regular cardinal such that $\lvert [\mu^+]^{<\mu} \rvert=\mu$.
  Let $\mathbb{P}(\mu)$ be the collection of all conditions $(s,f)$ such that:
  \begin{enumerate}[label={(\arabic*)}]
    \item $s\in [\mu^+\setminus \mu]^{<\mu}$\label{defpmu:s}
    \item $f:s\to [\mu^+\setminus \mu]^{<\mu}$ and if $\xi,\xi'\in s$ with $\xi<\xi'$ then $f(\xi)\subseteq\xi'$.\label{defpmu:f}
  \end{enumerate}
  We say $(s,f)\leq (t,g)$ if $s\supseteq t$ and whenever $\xi\in t$, $f(\xi)\supseteq g(\xi)$.
\end{definition}

For each $(s,f)\in \mathbb{P}(\mu)$, $s$ can be thought of as approximating $\dot{C_\mu}$, in the sense that $(s,f)\forces s\subseteq \dot{C_\mu}$
(in fact, we will later define $C_\mu=\bigcup_{(s,f)\in G} s$, for $G$ a $\mathbb{P}(\mu)$-generic filter over $V$).

Additionally, $f$ can be thought of as ``banning" certain ordinals from ever appearing in $\dot{C_\mu}$, in the sense that if $\alpha\in s$, $\beta>\alpha$, and $f(\alpha)\ni\beta$, then:
\begin{itemize}
\item it must be the case that $s\cap (\alpha,\beta]=\emptyset$. Otherwise, if $\gamma \in s\cap(\alpha,\beta]$, we would have that $\beta\in f(\alpha)$ and $\beta\notin\gamma$. Hence $f(\alpha)\not\subseteq\gamma$, contradicting conditionhood of $(s,f)$. 

\item Additionally, $(s,f)\forces \dot{C_\mu} \cap (\alpha,\beta]=\emptyset$. This is since for every $(t,g)\leq (s,f)$, $\beta\in g(\alpha)$; hence $t\cap (\alpha,\beta]=\emptyset$.

\end{itemize}

\begin{lemma}\label{pmuisnice}If $\mu$ is a regular cardinal, then $\mathbb{P}(\mu)$ has the following properties:
  \begin{enumerate}[label={(\arabic*)}]
    \item $|\mathbb{P}(\mu)|=\mu^+$ hence $\mathbb{P}(\mu)$ has the $\mu^{++}$-cc. \label{pmuisnice:cc}
    \item $\mathbb{P}(\mu)$ is $<\mu$-directed closed. \label{pmuisnice:dircl}
    \item If $\theta\geq\mu^{++}$, $M\prec (H_\theta,\in,\mu^+)$, and $M\cap \mu^+\in \mu^+\cap \cof(\mu)$, then $\mathbb{P}(\mu)$ is strongly proper for $M$. Hence $\mathbb{P}(\mu)$ preserves $\mu^+$. \label{pmuisnice:proper}
    \item If $G$ is $\mathbb{P}(\mu)$-generic over $V$, then in $V[G]$, we have that 
    \[C_\mu:=\bigcup_{(s,f)\in G} s\]
    is a club subset of $\mu^+$ such that if $X\in V$ and $|X|^V\geq \mu$, then $X\not\subseteq C_\mu$. \label{pmuisnice:club}
    \item $\mathbb{P}(\mu)$ is not $\mu^+$-cc below any condition. \label{pmuisnice:notcc}
  \end{enumerate}
\end{lemma}

\begin{proof}
The proofs are exactly as in Lemma 4.4 in \cite{cox_eskew_kill_sat_save_presat}.
For the sake of clarity, we will prove \ref{pmuisnice:proper} and \ref{pmuisnice:club}.

To see that \ref{pmuisnice:proper} holds, let $\theta\geq \mu^{++}$, $M\prec (H_\theta,\in,\mu^+)$, and $M\cap \mu^+\in \mu^+\cap \cof(\mu)$; suppose that $(s,f)\in \mathbb{P}(\mu)\cap M$.
Observe that $\mu^{<\mu}=\mu$ and $M\prec (H_\theta,\in,\mu^+,\mu)$.
Let $\delta=M\cap \mu^+$; since $(\mu^+)^{<\mu}=\mu^+$ as witnessed in $H_\theta$, we have that there is a bijection $\phi:\mu^+\to [\mu^+]^{<\mu}$ such that $\phi\in M$. 
Without loss of generality, we may assume that for each $\beta<\mu^+$ with $cf(\beta)=\mu$, $\phi\upharpoonright \beta$ surjects onto $[\beta]^{<\mu}$.

We wish to show that $^{<\mu} (M\cap \mu^+)\subseteq M$.
Let $\delta=M\cap \mu^+$ and suppose that $b\in [\delta]^{<\mu}$. 
Since $cf(\delta)=\mu$, we have that $\sup b < \delta$.
But then by choice of $\phi$, there is an $\alpha<\sup b$ such that $\phi(\alpha)=\beta$, and since $\sup b<\delta$, $\alpha\in M$.
Thus $b\in M$, and so we have shown
\[^{<\mu}(M\cap \mu^+)\subseteq M\]
Since $|s|<\mu\subseteq M\cap \mu^+$, we thus have that $s\subseteq M$ and hence $M\cap \mu^+\notin s=dom(f)$.
Further, if $\xi\in s$ then $f(\xi)\in M\cap [\mu^+]^{<\mu}$; since $\mu\subseteq M$ and $\theta$ is sufficiently large, $f(\xi)\subseteq M\cap\mu^+$.

Thus the following condition $(s',f')$ extends $(s,f)$:
\[(s',f'):=\left(s\frown \left(M\cap \mu^+\right),f\frown \left(M\cap \mu^+\mapsto \{M\cap \mu^+\}\right)\right)\]
We now must argue that $(s',f')$ is a strong master condition for $(M,\mathbb{P}(\mu))$. 
Let $(t,h)\leq (s',f')$. 
Then $t_M:=t\cap M$ is a $<\mu$-sized subset of $M\cap \mu^+$, hence $t_M\in M$.
Further, since $(t,h)\leq (s',f')$, we have that $M\cap\mu^+\in t$.
Hence, as $(t,h)$ is a condition in $\mathbb{P}(\mu)$ (namely, by part \ref{defpmu:f} of Definition \ref{defpmu}), $(h\upharpoonright t_M):t_M \to [M\cap \mu^+]^{<\mu}$.
Thus $(t_M,h\upharpoonright t_M)\in M\cap \mathbb{P}(\mu)$.

To complete the proof of strong properness, let $(u,g)\in M\cap\mathbb{P}(\mu)$, $(u,g)\leq (t_M,h\upharpoonright t_M)$.
Then let $F:u\cup t \to [\mu^+]^{<\mu}$, $F(\xi)=g(\xi)$ if $\xi\in u$, and $F(\xi)=h(\xi)$ otherwise.
Then $(u\cup t, F)\in\mathbb{P}(\mu)$ and $(u\cup t, F)\leq (u,g),(t,h)$. 
Since $(u,g)$ was arbitrary, we have shown that every extension of $(t_M,h\upharpoonright t_M)$ in $\mathbb{P}(\mu)\cap M$ is compatible with $(t,h)$.
Thus $(s',f')$ is a strong master condition. This completes our proof of \ref{pmuisnice:proper}.

To see that \ref{pmuisnice:club} holds, we have three things to show:
\begin{enumerate}[label={(\roman*)}]
\item $C_\mu$ is unbounded in $\mu^+$
\item $C_\mu$ is closed
\item If $X\in V$ and $|X|^V\geq\mu$ then $X\not\subseteq C_\mu$
\end{enumerate}

To see (i), let $(s,f)\in \mathbb{P}(\mu)$ and let $\alpha<\mu^+$. 
By definition of $\mathbb{P}(\mu)$, $|s|<\mu$ and for each $\beta\in s$, $f(\beta)$ is a $<\mu$-sized subset of $\mu^+$. 
Hence $\sup_{\beta\in s} \sup f(\beta)<\mu^+$, so let $\delta$ be such that $\sup_{\beta\in s} \sup f(\beta)<\delta<\mu^+$.
Then 
\[p:=(s\frown \delta, f \frown (\delta\mapsto \emptyset))\]
is a condition below $(s,f)$ such that $p\forces \delta\in \dot{C_\mu}$; thus $C_\mu$ is unbounded.

To see (ii), we argue contrapositively. Let $\beta\in \mu^+\setminus (\mu+1)$ and suppose $(s,f)\in\mathbb{P}(\mu)$ is such that $(s,f)\forces \check{\beta}\notin \dot{C_\mu}$. We will argue that $(s,f)\forces \check{\beta}\notin Lim(\dot{C_\mu})$.
Observe that there must be an $\alpha\in s\cap\beta$ such that $f(\alpha)\not\subseteq\beta$; for otherwise, we would have that for all $\alpha \in s\cap \beta$, $f(\alpha)\subseteq\beta$, hence $\left(s\frown \beta, f\frown(\beta\mapsto\emptyset)\right)$ would be a condition below $(s,f)$ forcing $\beta\in\dot{C_\mu}$.
By conditionhood of $(s,f)$, there is a unique such $\alpha$ and $\alpha$ is the largest element of $s\cap\beta$.
Additionally, no extension $(t,g)$ of $(s,f)$ can have that $t\cap (\alpha,\beta)\neq\emptyset$, and hence $(s,f)\forces``\check{\alpha}\text{ is the largest element of }\dot{C_\mu}\cap\check{\beta}"$.
Thus $(s,f)\forces\check{\beta}\notin Lim(\dot{C_\mu})$.

To see (iii), let $X\in V$ with $|X|^V\geq \mu$ and let $(s,f)\in \mathbb{P}(\mu)$. 
Observe that without loss of generality we may assume that $X\subseteq \mu^+\setminus (\mu+1)$.
Further, by taking an initial segment of $X$ we may assume that $otp(X)=\mu$ and hence that $cf(\sup(X))=\mu$.
Since $|s|<\mu$ and $\sup(X)$ has cofinality $\mu$, $s\cap \sup(X)$ is bounded below $\sup(X)$. 

Now we have two cases. If there is a $\xi\in s\cap \sup(X)$ such that $f(\xi)\not\subseteq \sup(X)$, let $\rho\in f(\xi)\setminus\sup(X)$.
Then $(s,f)\forces \dot{C_\mu}\cap (\xi,\rho]=\emptyset$ and hence $(s,f)\forces``\dot{C_\mu}\cap \check{X}\text{ is bounded below }\sup(\check{X})"$. 
Thus $X\not\subseteq C_\mu$.

Otherwise, let $\zeta=\sup\{\sup f(\xi) \mid \xi\in s\cap \sup(X)\}$.
Since each $f(\xi)\subseteq \sup(X)$ and $\mu$ is regular, $\zeta<\sup(X)$.
Let $p=(s\frown \zeta,f\frown (\zeta\mapsto\{\sup(X)\}))$.
Then $p\leq (s,f)$ and $p\forces \max(\dot{C_\mu}\cap \sup(X))=\zeta$.
Hence $p\forces X\not\subseteq \dot{C_\mu}$.
Thus $X\not\subseteq C_\mu$. This completes our proof of \ref{pmuisnice:club}.
\end{proof}

\begin{definition}
  We define an Easton support iteration forcing $\mathbb{Q}=\seq{\mathbb{Q}_\mu * \dot{\mathbb{C}}(\mu) \mid \mu<\kappa}$ as follows:
  
  For each $\mu<\kappa$, if $\mu$ is inaccessible in $V^{\mathbb{Q}_\mu}$, let $\mathbb{C}(\mu)=\mathbb{P}(\mu)$ as above, and otherwise let $\mathbb{C}(\mu)$ be the trivial forcing. 
\end{definition}

\begin{proposition}
If $\nu\leq\kappa$ is regular in $V$, then $\nu$ is still regular in $V^{\mathbb{Q}_\nu}$.\label{itn-regularity}
\end{proposition}

\begin{proof}
This breaks into three cases:
\begin{enumerate}
\item $\nu=\tau^+$, for $\tau$ a regular cardinal
\item $\nu=\lambda^+$, for $\lambda$ a singular cardinal
\item $\nu$ is inaccessible
\end{enumerate}

If $\nu=\tau^+$ where $\tau$ is regular, we may decompose $\mathbb{Q}_\nu$ as 
\[\mathbb{Q}_\tau * \dot{\mathbb{C}}(\tau)\]
Since $\tau$ is regular, $|\mathbb{Q}_\tau|\leq\tau$ hence is $\nu$-cc. Thus $\mathbb{Q}_\tau$ preserves $\nu$.
By Lemma \ref{pmuisnice}\ref{pmuisnice:proper}, $\dot{\mathbb{C}}(\tau)$ preserves $\nu$.
Thus $\dot{\mathbb{Q}}_{\nu}$ preserves $\nu$.

If $\nu=\lambda^+$ where $\lambda$ is singular, we have that $\mathbb{Q}_\nu=\mathbb{Q}_\lambda$ since none of the ordinals in $[\lambda,\nu)$ are regular. 
Here, the situation is more complicated, since now $|\mathbb{Q}_\lambda|=\lambda^{\cf(\lambda)}\geq\nu$. So we must verify more directly that $\nu$ is preserved.

So observe that if $\nu$ is collapsed, then $V^{\mathbb{Q}_\lambda}\models |\nu|\leq |\lambda|$ and since $\lambda$ is singular, we would have a $\mathbb{Q}_\lambda$-name $\dot{f}:\check{\delta}\to\check{\nu}$ for a cofinal sequence in $\check{\nu}$ for some regular cardinal $\delta<\lambda$.

But we may decompose $\mathbb{Q}_\lambda$ into 
\[\mathbb{Q}_{\delta}*\dot{\mathbb{C}}(\delta)*\dot{\mathbb{Q}}_{>\delta^+}\]
Now, $\dot{\mathbb{Q}}_{>\delta^+}$ is $<\delta^+$-directed closed, so $\dot{\mathbb{Q}}_{>\delta}$ could not have added such an $f$.
Additionally, $\dot{\mathbb{C}}(\delta)$ satisfies the $\delta^{++}$-cc, hence is $\nu$-cc. Thus $\dot{\mathbb{C}}(\delta)$ also could not have added $f$.
Finally, $|\mathbb{Q}_\delta|=\delta$ so $\mathbb{Q}_\delta$ satisfies the $\delta^+$-cc, hence is also $\nu$-cc. Thus $\mathbb{Q}_\delta$ could not have added such an $f$ either. 

As in the successor of a regular case, $\dot{\mathbb{C}}(\nu)$ and $\dot{\mathbb{Q}}_{\geq\nu}$ preserve $\nu$ as well.

And in the case where $\nu$ is inaccessible, suppose that in $V^{\mathbb{Q}_\nu}$ that $cf(\check{\nu})=\check{\delta}<\check{\nu}$.
Then $\mathbb{Q}_\nu$ decomposes, as in the successor of a singular case, into
$$\mathbb{Q}_\delta * \dot{\mathbb{C}}(\delta) * \dot{\mathbb{Q}}_{>\delta^+}$$
The analysis is exactly as in the successor of a singular case.
\end{proof}

This shows that whenever $\nu$ is regular in $V$, $\nu$ remains regular in $V^{\mathbb{Q}_\nu}$ and we will now write $\mathbb{P}(\nu)$ rather than $\mathbb{C}(\nu)$.

\begin{corollary}
  $\mathbb{Q}$ preserves cardinals.
\end{corollary}

\begin{proof}
Since $\kappa$ is Mahlo, $\mathbb{Q}=\mathbb{Q}_\kappa$  is $\kappa$-cc hence preserves $\kappa$ preserves cardinals $\geq\kappa$.

For $\nu<\kappa$ regular, we have that $\mathbb{Q}=\mathbb{Q}_\nu * \dot{\mathbb{P}}(\nu)*\dot{\mathbb{Q}}_{>\nu}$.
By the preceding proposition, $\mathbb{Q}_\nu$ preserves $\nu$. By Lemma \ref{pmuisnice}\ref{pmuisnice:proper}, $\dot{\mathbb{P}}(\nu)$ preserves $\nu$. And by Lemma \ref{pmuisnice}\ref{pmuisnice:dircl}, $\dot{\mathbb{Q}}_{>\nu}$ is $<\nu^+$-directed closed hence preserves $\nu$.
\end{proof}

\section{Destroying Saturation}
\label{killsat}
Since $\mathbb{Q}$ projects to each $\mathbb{Q}_\mu * \dot{\mathbb{P}}(\mu)$, $\mu<\kappa$ inaccessible, we may, for each such $\mu$, let $G_\mu$ be the restriction of the $\mathbb{Q}$-generic $G$ to $\mathbb{P}(\mu)$ and define 
$C_\mu=\{\xi \mid \exists (s,f)\in G_\mu \ \xi\in s\}$.
By Lemma \ref{pmuisnice}\ref{pmuisnice:club}, $C_\mu$ is a club subset of $\mu^+$ in $V^{\mathbb{Q}_{\mu}*\dot{\mathbb{P}}(\mu)}$ and for every $X\in V^{\mathbb{Q}_{\mu}}$ such that $X\subseteq [\mu,\mu^+)$ and $X$ has $V^{\mathbb{Q}_{\mu}}$-cardinality $\geq\mu$, $X\not\subseteq C_\mu$. 

As a warmup to proving Theorem \ref{bigthm}\ref{bigthm:killsat}, we argue the following:

\begin{proposition}\label{ground-killsat}
Suppose that $I\in V$ is $\kappa$-complete, normal, $\kappa^+$-saturated, and concentrates on $Reg_\kappa$. 
Then in $V^\mathbb{Q}$, $\overline{I}$ is not $\kappa^+$-saturated.
\end{proposition}

Before we prove this, it will be helpful to isolate a lemma on what $j_I(\mathbb{Q})$ looks like in $\Ult(V,I)$:

\begin{lemma}\label{jQ-ultVI}
Let $I$ be a $\kappa$-complete, normal, fine precipitous ideal concentrating on inaccessibles.
Then in $\Ult(V,I)$, 
$j_I(\mathbb{Q})\cong \mathbb{Q}*\dot{\mathbb{R}}$, where $\dot{\mathbb{R}}$ is a name for an Easton support iteration $\seq{\mathbb{R}_\alpha * \dot{\mathbb{C}}(\alpha) \mid \alpha\in[\kappa,j(\kappa))}$, such that if $\alpha$ is inaccessible, $\mathbb{C}(\alpha)=\mathbb{P}(\alpha)$, and $\mathbb{C}(\alpha)$ is the trivial forcing otherwise.
\end{lemma}

\begin{proof}
This follows from the elementarity of $j_I$, and since $I$ concentrates on regulars, $\kappa$ is regular in $\Ult(V,I)$.
\end{proof}

And we remark on how $\Ult(V,I)$ computes $\kappa^+$:

\begin{lemma}
  Let $I$ be as in Lemma \ref{jQ-ultVI}, with $G$ a $\mathcal{B}_I$-generic over $V$.
  Then if $I$ is $\kappa^+$-saturated, then $(\kappa^+)^{\Ult(V,I)}=(\kappa^+)^{V[G]}$.
\end{lemma}

Thus for $\kappa^+$-saturated ideals, we will just write $\kappa^+$ to mean $(\kappa^+)^V=(\kappa^+)^{V[G]}=(\kappa^+)^{\Ult(V,I)}$.

\begin{proof}
  Since $I$ is $\kappa^+$-presaturated, forcing with $\mathcal{B}_I$ preserves $\kappa^+$, hence $(\kappa^+)^V=(\kappa^+)^{V[G]}\geq (\kappa^+)^{\Ult(V,I)}$.
  
  To see that $(\kappa^+)^{V[G]}=(\kappa^+)^{\Ult(V,I)}$, 
  suppose that $\alpha\in [\kappa,(\kappa^+)^{V[G]})$.
  Let this be witnessed by some $\kappa$-sequence $f:\kappa\to\alpha$ a bijection.
  Since $I$ is $\kappa^+$-saturated, we have from Fact \ref{foreman_ideals_gen_ees:ultrapower-closure} that $\Ult(V,I)$ is closed under $\kappa$-sequences from $V[G]$, and so $f\in \Ult(V,I)$.
  Thus $\alpha$ is not a cardinal in $\Ult(V,I)$.
\end{proof}

\begin{remark}
As with ultrapowers from a measurable cardinal, we will have that if $I$ is a $\kappa$-complete normal precipitous ideal in $V$, then in $V^{\mathcal{B}_I}$, $|j_I(\kappa)|=2^\kappa$.
However, by elementarity, in $\Ult(V,I)$, $j_I(\kappa)$ is inaccessible.
\end{remark}

\begin{remark}
This is unlike a $\lambda$-complete, $\lambda^+$-saturated ideal $J$ on $\lambda$ a successor cardinal;
for $\lambda$ a successor cardinal, we would have that $j_J(\lambda)=\lambda^+$. The argument can be found in \cite{foreman_ideals_gen_ees}.
\end{remark}

\begin{proof}[Proof of Proposition \ref{ground-killsat}]
By Lemma \ref{jQ-ultVI}, in $V^{\mathcal{B}_I}$, $\dot{j}_I(\mathbb{Q})\cong \mathbb{Q}*\dot{\mathbb{R}}$, where $\dot{\mathbb{R}}$ is an Easton support iteration $\seq{\mathbb{R}_\alpha * \dot{\mathbb{C}}(\alpha) \mid \alpha\in[\kappa,j_I(\kappa))}$ as in the lemma.

Since $I$ concentrates on inaccessibles below $\kappa$,
$\kappa$ is still inaccessible in $\Ult(V,I)$.
Thus $\mathbb{C}(\kappa)=\mathbb{P}(\kappa)$ which is not $\kappa^+$-cc. 
So $j_I(\mathbb{Q})$ is not $\kappa^+$-saturated.

So by Lemma \ref{foreman_ideals_gen_ees:groundideals}, in $V^\mathbb{Q}$, $\overline{I}$ is not $\kappa^+$-saturated.
\end{proof}

We now prove Theorem \ref{bigthm}\ref{bigthm:killsat}.

\begin{proof}[Proof of Theorem \ref{bigthm}\ref{bigthm:killsat}]
Let $G$ be $\mathbb{Q}$-generic, and suppose that in $V[G]$ there is a $\kappa$-complete, $\kappa^+$-saturated ideal $\mathcal{J}$ on $\kappa$ concentrating on regular cardinals below $\kappa$.

Let $U$ be $P(\kappa)/\mathcal{J}$-generic over $V[G]$, and let $j:V[G]\to \Ult(V[G],U)$ be the generic ultrapower.

Let $N=\bigcup_{\alpha\in ORD} j(V_\alpha)$. Then $j(\mathbb{Q})\in N$ and hence $\Ult(V[G],U)=N[g']$ for some $g'\in V[G*U]$ which is $j(\mathbb{Q})$-generic over $N$.

Observe that $\kappa$ is still inaccessible in $N[g']$ by inaccessibility in $V[G]$, by being the critical point of $j$, and since $\mathcal{J}$ concentrates on regulars.
Since $j(\kappa)>\kappa$ and $j(\kappa)$ is a cardinal in $N[g']$, $j(\kappa)> (\kappa^+)^{N[g']}\geq (\kappa^+)^{V[G]}$ (by $\kappa$-closure and $\kappa^+$-saturation of $\mathcal{J}$).
Further, by the usual ultrapower argument, $|j(\kappa)|=2^\kappa$. 

So $j(\kappa)$ is not a cardinal in $V$, but by Fact \ref{foreman_ideals_gen_ees:ultrapower-closure}, $N[g']$ is closed under $\kappa$-sequences from $V[G]$.

Work in $N[g']$. Let $g'$ be the projection of $j(\mathbb{Q})$ to $\mathbb{P}(\kappa)$, and let 
\[C_\kappa=\bigcup_{(s,f)\in g'} s\]
Then
\begin{equation}\label{eqn:ng'defc}
N[g']\models C_\kappa\text{ is club in }\kappa^+ \text{ and } \forall  X \in N |X|^N\geq\kappa, \ X \not\subseteq C_\kappa
\end{equation}

Since $V[G*U]$ is a $\kappa^+$-cc extension of $V$, we may let $D\in V$ be such that in $V[G*U]$, $D$ is a club subset of $C_\kappa$. 
Let $E\subseteq D$ be in $V$, $(o.t.(E))^V=\kappa$, $\alpha=\sup E$; since $cf(\alpha)=\kappa$, let $\phi:\kappa\to\alpha$ be a normal increasing sequence.

Let $E'=lim(E)\cap ran(\phi)$.
Then $E'\subseteq D$ and $|E'|^V=\kappa$ since $\kappa$ is inaccessible.
Further, $j(\phi)\in N$ and $j(\phi)\upharpoonright \kappa:\kappa\to j"\alpha$ is also in $N$.
Thus $ran(j(\phi)\upharpoonright \kappa)\in N$ and $j"E'\subseteq ran(j(\phi)\upharpoonright \kappa)\subseteq j"\alpha$.

But $j"E'=ran(j(\phi)\upharpoonright \kappa) \cap j(E')\in N$; and since $E'=\{\beta\in ran(\phi)\mid j(\beta)\in j(E')\}$, we have that $E'$ is a subset of $C_\kappa$ with $|E'|^N=\kappa$ and $E'\subseteq[\kappa,\kappa^+)$. 

This contradicts Statement (\ref{eqn:ng'defc}), and hence $\mathcal{J}$ cannot be $\kappa^+$-saturated.
\end{proof}

\section{Preserving Presaturation}
\label{savepresat}
We now prove Theorem \ref{bigthm}\ref{bigthm:savepresat}.

\begin{proof}[Proof of Theorem  \ref{bigthm}\ref{bigthm:savepresat}]
Let $I\in V$ be a $\kappa$-complete, normal, $\kappa^+$-saturated ideal in $V$ concentrating on inaccessibles.
Work in $V^{\mathcal{B}_I}$ and let $U$ be the generic ultrafilter. Since $I$ is $\kappa$-complete, $crit(\dot{j}_I)=\kappa$ and $\dot{j}_I\upharpoonright \kappa=id$.

Thus, in $\Ult(V,U)$, by Lemma \ref{jQ-ultVI}, $\dot{j}_I(\mathbb{Q})\cong \mathbb{Q}*\dot{\mathbb{P}(\kappa)}*\dot{\mathbb{R}}$, where $\dot{\mathbb{R}}$ is an Easton support iteration $\seq{\mathbb{R}_\alpha * \mathbb{C}(\alpha) \mid \alpha\in[\kappa^+,j_I(\kappa))}$, such that if $\alpha$ is inaccessible, $\mathbb{C}(\alpha)=\mathbb{P}(\alpha)$, and $\mathbb{C}(\alpha)$ is the trivial forcing otherwise.

We will argue that $\mathcal{B}_I * \dot{j}_I(\mathbb{Q})$ is $\kappa^+$-proper on a stationary set, and hence is $\kappa^+$-presaturated.

Observe that $\mathcal{B}_I$ is $\kappa^+$-cc. Hence, in $V[U]$ so also in $\Ult(V,U)$, $\mathbb{Q}$ is still $\kappa^+$-cc. 
Thus, in $\Ult(V,U)$, $\mathcal{B}_I * \mathbb{Q}$ is $\kappa^+$-cc and hence is $\kappa^+$-proper on $\mathcal{P}^*_{\kappa^+}(H_\theta)$ for all sufficiently large $\theta$.

The difficulty comes in assuring $\mathbb{P}(\kappa)*\dot{\mathbb{R}}$ preserves the properness on a stationary set.
We will do this by arguing that $\mathbb{P}(\kappa) * \dot{\mathbb{R}}$ is forced by $\mathcal{B}_I*\mathbb{Q}$ to be $\kappa^+$-proper on $\dot{IA_{<\kappa^+}}$.
Once we have that, since $\mathcal{B}_I*\mathbb{Q}$ is $\kappa^+$-cc and forces $\mathbb{P}(\kappa)* \dot{\mathbb{R}}$ is $\kappa^+$-proper on a stationary set, the full forcing $\mathcal{B}_I * \dot{j}_I(\mathbb{Q})$ is then $\kappa^+$-proper on a stationary set.

Work in $\Ult(V,U)^{\mathbb{Q}}$. Here, $\mathbb{P}(\kappa)$ is proper on 
$$\mathcal{S}:={\{M\prec (H_\theta,\in,\kappa^+)\mid |M|=|M\cap\kappa^+|=\kappa\text{ and }M\cap\kappa^+\in \cof(\kappa)\}}$$
and by the $<\kappa^+$-directed closedness of $\dot{\mathbb{R}}$ and Fact \ref{principle-A}, $\forces_{\mathbb{P}(\kappa)}\dot{\mathbb{R}}\text{ proper on }\check{IA}_{<\kappa^+}$.
But not only is $\mathcal{S}$ stationary, $\mathcal{S}$ is a club subset of $\mathcal{P}^*_{\kappa^+}(H_\theta) \upharpoonright \cof(\kappa)$, and hence $\mathcal{S}\cap IA_{<\kappa^+}$ is also stationary.


Thus, by Lemma \ref{lemma-iaposets}, after forcing with $\mathcal{B}_I * \dot{j}_I(\mathbb{Q})$, we have that $\mathbb{P}(\kappa)*\dot{\mathbb{R}}$ is $\kappa^+$-proper on the stationary set $\mathcal{S}\cap IA_{<\delta^+}$.

Therefore, $\mathcal{B}_I * \dot{j}_I\left(\mathbb{Q}\right)$ is $\kappa^+$-proper on a stationary subset of $\mathcal{P}_{\kappa^+}^*(H_\theta)^V$, hence is $\kappa^+$-presaturated.
But by Theorem \ref{easyduality}, $\mathcal{B}_I * \dot{j}_I\left(\mathbb{Q}\right)\cong \mathbb{Q}*\dot{\mathcal{B}_{\overline{I}}}$;
then by Lemma \ref{pushpresatdown}, $\overline{I}$ is $\kappa^+$-presaturated.
\end{proof}

A more general argument will prove Theorem \ref{savemorepresat}:

\begin{proof}[Proof of Theorem \ref{savemorepresat}]
In $V$, let $\delta\geq \kappa$ be an inaccessible cardinal, and let $I$ be a normal, precipitous, fine, $\delta^+$-presaturated ideal of uniform completeness $\kappa$ on some algebra of sets $Z$ such that $\mathcal{B}_I$ preserves the inaccessibility of $\kappa$ and $\delta$ in $\Ult(V,I)$; $\forces_{\mathcal{B}_I}\delta^+\leq |\dot{j}_I(\kappa)|<\dot{j}(\kappa)$; and $\mathcal{B}_I$ is $\delta^+$-proper on $IA_{<\delta^+}$.

We wish to show that in $V^{\mathbb{Q}}$, $\mathcal{B}_{\overline{I}}$ is not $\delta^+$-saturated, but is $\delta^+$-presaturated.

Recall that $\mathbb{Q}$ is $\kappa$-cc since $\mathbb{Q}$ is an Easton support iteration of $\kappa$-cc posets and $\kappa$ is Mahlo.
Since $\mathcal{B}_I$ is precipitous, by Theorem \ref{easyduality},
\begin{equation} \label{eqn:savemorepresat-duality}
  \mathcal{B}\left(\mathcal{B}_I * \dot{j}_I(\mathbb{Q})\right) \cong \mathcal{B}\left(\mathbb{Q} * \mathcal{B}_{\overline{I}}\right)
\end{equation}
Also, since $I$ is $\kappa$-complete, $crit(j_I)=\kappa$ and thus $j_I \upharpoonright \kappa=id$.
Since $\mathcal{B}_I$ preserves the regularity of $\kappa$, we get that $\dot{j}_I(\mathbb{Q})\upharpoonright \kappa=\mathbb{Q}$.

Therefore, $j_I(\mathbb{Q})=\mathbb{Q} * \seq{\mathbb{R}_\alpha * \mathbb{C}(\alpha) \mid \alpha\in[\kappa,j_I(\kappa))}$, where each $\mathbb{R}_\alpha$ is an Easton support iteration,
such that if $\alpha$ is inaccessible, $\mathbb{C}(\alpha)=\mathbb{P}(\alpha)$, and $\mathbb{C}(\alpha)$ is the trivial forcing otherwise.

Since $\Ult(V,I)\models ``\delta \text{ inaccessible}"$, we get that $\mathbb{C}(\delta)=\mathbb{P}(\delta)$ which is not $\delta^+$-cc.
Thus $\left(\mathbb{Q} * \mathcal{B}_{\overline{I}}\right)$ is not $\delta^+$-cc, and since $\mathbb{Q}$ is clearly $\delta^+$-cc, $\mathcal{B}_{\overline{I}}$ cannot be $\delta^+$-saturated.

As for the $\delta^+$-presaturation of $\mathcal{B}_{\overline{I}}$, by \ref{eqn:savemorepresat-duality} it suffices to show that  $\mathcal{B}_I * \dot{j}_I(\mathbb{Q})$ is $\delta^+$-presaturated.

Work in $\Ult(V,I)$.
Since $\mathcal{B}_I$ preserves the regularity of $\kappa$, we decompose $\dot{j}_I(\mathbb{Q})$ as 
\[
  \dot{j}_I(\mathbb{Q})=\mathbb{Q} * \left(\dot{j}_I(\mathbb{Q}) \upharpoonright [\kappa,\delta) \right) * \left( \dot{j}_I(\mathbb{Q})\upharpoonright [\delta,j_I(\kappa))\right)
\]
and further $\dot{j}_I(\mathbb{Q})(\delta)=\dot{\mathbb{P}}(\delta)$, so we decompose $\left( \dot{j}_I(\mathbb{Q})\upharpoonright [\delta,j_I(\kappa))\right)$ as $\left( \dot{\mathbb{P}}(\delta) * \dot{j}_I(\mathbb{Q})\upharpoonright [\delta^+,j_I(\kappa))\right)$.

So we may further decompose $\dot{j}_I(\mathbb{Q})$ as
\[
  \dot{j}_I(\mathbb{Q})=\mathbb{Q} * \left(\dot{j}_I(\mathbb{Q}) \upharpoonright [\kappa,\delta) \right) *
  \dot{\mathbb{P}}(\delta) * \left( \dot{j}_I(\mathbb{Q})\upharpoonright [\delta^+,j_I(\kappa))\right)
\]
and we will argue the following items in $\Ult(V,I)$:
\begin{itemize}
  \item $\mathbb{Q}$ is $\delta^+$-cc
  \item $\dot{j}_I(\mathbb{Q}) \upharpoonright [\kappa,\delta)$ is $\delta^+$-cc
  \item $\dot{\mathbb{P}}(\delta)$ is $\delta^+$-proper on a stationary set $S$ such that $S \cap IA_{<\delta^+}$ is stationary
  \item $\left( \dot{j}_I(\mathbb{Q})\upharpoonright [\delta^+,j_I(\kappa))\right)$ is $\delta^+$-directed closed and thus is $\delta$-proper on $IA_{<\delta^+}$
\end{itemize}
in such a way that we may conclude that $\mathcal{B}_I * \dot{j}_I(\mathbb{Q})$ is $\delta^+$-proper on a stationary set.

We have that $\mathbb{Q}$ is $\kappa$-cc, hence is $\delta^+$-cc.

If $\delta=\kappa$, then $\left(\dot{j}_I(\mathbb{Q}) \upharpoonright [\kappa,\delta) \right)$ is trivial so is $\delta^+$-cc.
Otherwise, $\delta>\kappa$, and since $\delta$ is inaccessible in $\Ult(V,I)$,
$\left(\dot{j}_I(\mathbb{Q}) \upharpoonright [\kappa,\delta) \right)$ is a $\delta$-length direct limit iteration of posets of size $<\delta$.
Thus $\left(\dot{j}_I(\mathbb{Q}) \upharpoonright [\kappa,\delta) \right)$ has size $\delta$, and so is $\delta^+$-cc.

Therefore $\mathcal{B}_I * \mathbb Q * \left(\dot{j}_I(\mathbb{Q}) \upharpoonright [\kappa,\delta) \right)$ is $\delta^+$-cc, so by Lemma \label{iaposets}\label{iaposets:iter}, is $\delta^+$-proper on $\mathcal{P}_{\delta^+}^*(H_\theta)$.

As in the proof of Theorem \ref{bigthm}\ref{bigthm:savepresat}, $\mathbb{P}(\delta)$ is proper on a club subset of $\mathcal{P}_{\delta^+}^*(H_\theta) \cap \cof(\delta)$ and so $\mathcal{B}_I * \mathbb Q * \left(\dot{j}_I(\mathbb{Q}) \upharpoonright [\kappa^+,\delta) \right) * \mathbb{P}(\delta)$ is $\delta^+$-proper on the stationary set $IA_{<\delta^+} \cap \cof(\delta)$.
Finally, $\left( \dot{j}_I(\mathbb{Q})\upharpoonright [\delta^+,j_I(\kappa))\right)$ is $\delta^+$-directed closed and therefore is proper on $IA_{<\delta^+}$, which by Fact \ref{principle-A}, is absolute between $V^{\mathcal{B}_I}$ and $V^{\mathcal{B}_I * \mathbb Q * \left(\dot{j}_I(\mathbb{Q}) \upharpoonright [\kappa^+,\delta) \right) * \mathbb{P}(\delta)}$.

Thus $\mathcal{B}_I * \dot{j}_I(\mathbb{Q})$ is $\delta^+$-proper on the stationary subset $IA_{<\delta^+}\cap \cof(\delta)$, and therefore is $\delta^+$-presaturated.
But then $\mathbb{Q} * \mathcal{B}_{\overline{I}}$ is $\delta^+$-presaturated as well, and therefore by Lemma \ref{pushpresatdown}, $V^{\mathbb{Q}}\models ``\mathcal{B}_{\overline{I}}\text{ is }\delta^+\text{-presaturated}"$.
\end{proof}

\section{Conclusions and Questions}
\label{open-qs}
We thus have that in $V^{\mathbb{Q}}$, certain $\kappa^+$-saturated ideals on $\kappa$ in $V$ are no longer $\kappa^+$-saturated, but remain $\kappa^+$-presaturated. 
As a consequence, we have counterexamples to Question \ref{foremanq13} at inaccessible cardinals.

Using Fact \ref{cof->presat}, Cox and Eskew argued in \cite{cox_eskew_kill_sat_save_presat} that their forcing $\mathbb{P}$ preserved the $\kappa^+$-presaturation of a much larger class of ideals on $\kappa$; this was possible because in their context, $j_I(\mathbb{P})$ was $\delta^{+\omega}$-cc. This naturally leads to the following question

\begin{question}
Does $\mathbb{Q}$ preserve the $\delta$-presaturation of all $\delta$-presaturated ideals on $\kappa$ concentrating on regular cardinals?
\end{question}

However, for us, $j_I(\mathbb{Q})$ will not be $\delta^{+\omega}$-cc, so Fact \ref{cof->presat} does not apply. 
This is why we only show that ideals that are $\delta$-proper on $IA_{<\delta}$ remain $\delta$-presaturated; $\delta$-saturated ideals are $\delta$-proper on $IA_{<\delta}$, so this was sufficient for our purposes.
We would need more powerful tools to argue that all $\kappa^+$-presaturated ideals in $V$ remain $\kappa^+$-presaturated in $V^{\mathbb{Q}}$.

Perhaps the more significant open questions are:

\begin{question}
Does $\mathbb{Q}$ preserve the $\delta$-presaturation of all $\delta$-presaturated ideals on $\kappa$?
\end{question}

The difficulty in this question, beyond the issues already raised, lies in the fact that here may be $\kappa^+$-saturated ideals at an inaccessible concentrating on $\cof(\mu)$, for some $\mu<\kappa$; 
in particular, as per \cite{gitik_chg_cofs_ns}, it is consistent that this is true for $NS_\kappa \upharpoonright S$, where $S$ is some stationary subset of $\kappa \cap \cof(\mu)$.

Forcing with such an ideal $I$ singularizes $\kappa$, and hence the following complications arise: 

\begin{enumerate}
  \item $j_I(\mathbb{Q})\upharpoonright \kappa$ will be an inverse limit, so it is not clear that $j_I(\mathbb{Q})$ will be proper on a stationary subset of $\mathcal{P}^*_\kappa(\kappa^+)$ as would be required to mimic the above arguments.

  \item $j_I(\mathbb{Q})(\kappa)$ will be trivial, so the iteration would need to be modified to be active at singular stages.
\end{enumerate}

The second complication motivates the following questions:

\begin{question}
Is there a forcing that preserves the above presaturation but additionally forces there to be no $\delta$-saturated ideals at all?
\end{question}

Doing so may be feasible if $j_I(\mathbb{Q})(\kappa)$ is nontrivial and adds a club as in Definition \ref{defpmu}:

\begin{question}
Let $\lambda$ be singular and assume GCH. 
Is there a cardinal preserving, $\lambda^+$-proper on a stationary set (or even just $\lambda^+$-presaturated) poset $\mathbb{P}s(\lambda)$ that adds a club subset $C$ of $\lambda^+$ such that whenever $X\in V$, $|X|^V\geq\lambda$, we have that $X\not\subseteq C$?
\end{question}

Attempting to tackle Questions 6.2 warrants using more general versions of Foreman's Duality Theorem that apply even to non-$\kappa$-cc forcings; these more powerful formulations can be found in \cite{foreman_ideals_gen_ees}.

\section{Acknowledgments}
\label{acknowledgments}
The author would like to thank Dima Sinapova, his thesis advisor, for many helpful discussions, and Sean Cox and Monroe Eskew for guidance and focus on extending their results.

It is not clear whether generalizations of Theorems \ref{bigthm}\ref{bigthm:killsat}, \ref{bigthm}\ref{bigthm:savepresat}, and \ref{savemorepresat} apply for ideals not concentrating on regulars; the author wishes to thank the referee for pointing out this problem in earlier versions of the paper, among other corrections.

\bibliographystyle{plain}
\bibliography{main}
\end{document}